\documentclass[12pt,a4paper]{article}

\usepackage{graphicx,tikz}
\usepackage{latexsym}
\usepackage{amsmath}
\usepackage{amssymb}
\usepackage{array}
\usepackage{arydshln}

\usepackage{cite}
\usepackage{stmaryrd}
\usepackage{enumerate}

\usepackage{hyperref}

\usepackage{color}
\usepackage{lineno}
\usepackage{graphicx}
\usepackage{ae}
\usepackage{amsmath}
\usepackage{amssymb}
\usepackage{latexsym}
\usepackage{url}
\usepackage{epsfig}
\usepackage{mathrsfs}
\usepackage{amsfonts}
\usepackage{amsthm}
\usepackage{bm}

\newtheorem{theorem}{Theorem}[section]

\newtheorem{lemma}[theorem]{Lemma}

\theoremstyle{definition}

\newtheorem{problem}{Problem}

\numberwithin{equation}{section} 
\allowdisplaybreaks    

\def\qed{\hfill$\Box$\vspace{12pt}}

\long\def\delete#1{}

\usepackage{xcolor}
\usepackage[normalem]{ulem}

\usepackage{stfloats}

\begin{document}
\title{The asymptotic uniform distribution of subset sums}
\author{
~Jing Wang$^{a,b,c}$$^,$\thanks{Corresponding author. Email: wj66@mail.nwpu.edu.cn.}\\[2mm]
{\small $^a$School of Mathematics and Statistics,}\\[-0.8ex]
{\small Northwestern Polytechnical University, Xi'an, Shaanxi 710072, P.R.~China}\\[-0.8ex]
{\small $^b$Research \& Development Institute of Northwestern Polytechnical University in Shenzhen,}\\[-0.8ex]
{\small Shenzhen, Guangdong 518063, P.R. China}\\[-0.8ex]
{\small $^c$Xi'an-Budapest Joint Research Center for Combinatorics,}\\[-0.8ex]
{\small Northwestern Polytechnical University, Xi'an, Shaanxi 710129, P.R. China}\\
}

\date{}

\openup 0.5\jot
\maketitle
\begin{abstract}
Let $G$ be a finite abelian group of order $n$, and for each $a\in G$ and integer $1\le h\le n$ let $\mathcal{F}_a(h)$
denote the family of all  $h$-element subsets of $G$ whose sum is $a$.
A problem posed by Katona and Makar-Limanov is to determine whether the minimum and maximum sizes of the families $\mathcal{F}_a(h)$ (as $a$ ranges over $G$)  become asymptotically equal as $n\rightarrow \infty$ when $h=\left\lfloor\frac{n}{2}\right\rfloor$. 
We affirmatively answer this question and in fact show that the same asymptotic equality holds for every $4\leq h\leq \left\lfloor\frac{n}{2}\right\rfloor+1$.


\end{abstract}

\section{Introduction}

Let \([n] = \{1,2,\dots,n\}\) be a finite set, and let \(\mathcal{F} \subset 2^{[n]}\) be a family of its subsets. A  result in extremal combinatorics, \emph{Sperner's theorem} \cite{Sperner1928}, states that if \(\mathcal{F}\) contains no comparable pairs \(F_1 \subset F_2\) (\(F_1, F_2 \in \mathcal{F}\)), then  
\[
|\mathcal{F}| \leq \binom{n}{\lfloor \frac{n}{2} \rfloor}.
\]
This bound is sharp, as evidenced by the family consisting of all subsets of size \(\lfloor \frac{n}{2} \rfloor\).  

There are numerous extensions and variations of Sperner's theorem. One such result, due to Katona and Tarj\'{a}n \cite{KatonaT}, considers the case where \(\mathcal{F}\) contains no \emph{$2$-fork}, meaning no subset \(F\) that is included in two distinct sets \(F_1, F_2\) (\(F, F_1, F_2 \in \mathcal{F}\)). They established the following bounds for the maximum size of \(\mathcal{F}\):  
\[
\binom{n}{\lfloor \frac{n}{2} \rfloor} \left(1 + \frac{1}{n} + \Omega\left(\frac{1}{n^2}\right)\right) 
\leq \max |\mathcal{F}| 
\leq \binom{n}{\lfloor \frac{n}{2} \rfloor} \left(1 + \frac{2}{n}\right).
\]
Despite this progress, there remains a gap between the lower and upper bounds, which motivates further study.

There exists a construction for a family that forbids  $2$-fork. The construction consists of all subsets of size \(\lfloor \frac{n}{2} \rfloor\), along with an additional family \(\mathcal{F}(\lfloor \frac{n}{2} \rfloor +1)\) of \((\lfloor\frac{n}{2} \rfloor +1)\)-element subsets, which satisfies the following intersection property:
\begin{equation}\label{Aabeliangroup1}
|F_1 \cap F_2| < \left\lfloor \frac{n}{2} \right\rfloor \quad \text{for every } F_1, F_2 \in \mathcal{F}\left(\left\lfloor \frac{n}{2} \right\rfloor +1\right).
\end{equation}
This condition ensures that no subset \(F\) is contained in two distinct sets \(F_1, F_2\), thereby satisfying the required restriction.

A key approach to constructing the family \(\mathcal{F}(\lfloor \frac{n}{2} \rfloor +1)\) satisfying the above restriction is inspired by an old problem in \emph{coding theory}. The \emph{Hamming distance} between two binary sequences of length \(n\) is the number of positions in which they differ. A \emph{code of Hamming distance $d$} is a set of binary sequences of length \(n\) with a prescribed minimum Hamming distance \(d\). In particular, a \(h\)-\emph{constant-weight code} is one in which each sequence has exactly \(h\) ones. The problem of determining the largest size of a \(h\)-constant-weight code with Hamming distance of at least 4, denoted as \(C(n, h, 4)\), is equivalent to determining the largest family \(\mathcal{F}(h)\) satisfying (\ref{Aabeliangroup1}) when \(h = \lfloor \frac{n}{2} \rfloor + 1\).

The best-known bounds for \(\max |\mathcal{F}(\lfloor \frac{n}{2} \rfloor + 1)|\) were established in \cite{GrahamS1980, KatonaM2008}:  
\begin{equation}\label{Aabeliangroup0}
\frac{1}{n} \binom{n}{\lfloor \frac{n}{2} \rfloor+1} \leq \max |\mathcal{F}\left(\left\lfloor \frac{n}{2} \right\rfloor+1\right)| \leq \frac{2}{n} \binom{n}{\lfloor \frac{n}{2} \rfloor}.
\end{equation}
The factor of 2 between the lower and upper bounds remains an unresolved issue. In this paper, we study a  construction serving as a lower bound.

Let \(G\) be a finite abelian group of order \(n\). Define  the family
\[
\mathcal{F}_a(h) = \{ \{x_1, \dots, x_h\} \mid x_1, \dots, x_h \in G \text{ are distinct and } x_1 + \cdots + x_h = a\},
\]
for \(a\in G\) and $1\leq h\leq n$.
This construction, introduced in \cite{GrahamS1980} and further studied in \cite{KatonaM2008}, ensures that every two sets in \(\mathcal{F}_a(\lfloor \frac{n}{2} \rfloor +1)\) satisfy (\ref{Aabeliangroup1}).
 For more applications, the readers can refer to 
\cite{BoseRao1978,VarshamovT1965, DeBonisK2007,KatonaFor2008,Gerbner2018}.
Katona and Makar-Limanov \cite{KatonaM2008} posed the following question and they proved that has a positive solution for the group $\mathbb{Z}_2^r$.

\begin{problem}\label{Aabeliangroup}
\cite{KatonaM2008} Let \(G\) be a finite abelian group of order \(n\) and let \(h = \lfloor \frac{n}{2} \rfloor\). Is it always true that  
\[
\lim\limits_{n\rightarrow \infty}\frac{\min\limits_{a\in  G}|\mathcal{F}_a(h)|}{\max\limits_{a\in  G}|\mathcal{F}_a(h)|}=1?
\]
\end{problem}
Actually, Wagon and Wilf \cite{wagonwilf1994} also asked when the sums of the $h$-element subset are distributed uniformly in general. In this paper, we affirmatively answer Problem \ref{Aabeliangroup} and in fact show that the same asymptotic equality holds for every $4\leq h\leq \left\lfloor\frac{n}{2}\right\rfloor+1$. Moreover, since
$$|\mathcal{F}_a(h)|=|\mathcal{F}_{\sum\limits_{x\in G}x-a}(n-h)|,$$
the asymptotic uniform distribution extends to $\lfloor \frac{n}{2} \rfloor +1<h\leq n-4$ as well.
Very recently, Pach \cite{pachpp2025} also independently solved the problem using Kronecker-delta functions,  which differs from  our approach.

\begin{theorem}\label{Aabeliangroup-theo}
Let \(G\) be a finite abelian group of order \(n\) and let $h:=h(n)$ be a function of $n$. Then 
\[
\lim\limits_{n\rightarrow \infty}\frac{\min\limits_{a\in  G}|\mathcal{F}_a(h)|}{\max\limits_{a\in  G}|\mathcal{F}_a(h)|}=1
\]
holds for all \(4\leq h\leq \lfloor \frac{n}{2} \rfloor +1\).
\end{theorem}
It is obvious that
$$\sum\limits_{a\in G}|F_a(h)|=\binom{n}{h}.$$ 
Theorem \ref{Aabeliangroup-theo} establishes the asymptotic uniformity of the sizes of $\mathcal{F}_a(h)$  for different values of \(a\), that is, for all $a\in G$, $$|\mathcal{F}_a(h)|\rightarrow \frac{1}{n}\binom{n}{h}$$ 
holds as $n\rightarrow \infty$ for $4\leq h\leq \lfloor \frac{n}{2} \rfloor +1$.  That also implies that the family $\mathcal{F}_a(h)$ is a construction serving as a lower bound of (\ref{Aabeliangroup0}) for $a\in G$ when $h=\lfloor \frac{n}{2} \rfloor +1$ and $n\rightarrow \infty$.

\section{The proof of Theorem \ref{Aabeliangroup-theo}}
For clarity, we define  $f_a(h)$ as the size of the family $\mathcal{F}_a(h)$.
Let $g^x_{a}(h,i)$ denote the number of subsets $\{x_1,\ldots,x_{h-i}\}\subseteq  G$ satisfying $x_1+\cdots+x_{h-i}=a-ix$ with $x\in \{x_1,\ldots,x_{h-i}\}$ for $1\leq i\leq h-1$.  In particular, for $i=h-1$,  $g^x_{a}(h,h-1)$ represents 1 if  $hx=a$ and $0$ otherwise.
\begin{lemma}\label{Acyclicgroup-lemma2}
 Let $h\leq n-1$ and $a\in G$. Then
    \begin{align*}\nonumber
        f_a(h)=
        \frac{1}{h}&\left(\binom{n}{h-1}-\sum\limits_{\substack{i=2}}^{h-1}(-1)^{i}\sum\limits_{x\in G}f_{a-ix}(h-i)+(-1)^{h-1}\sum\limits_{x\in G}g^x_{a}(h,h-1)\right).
    \end{align*}
\end{lemma}

\begin{proof}
Observe that for any $ F_1, F_2 \in \mathcal{F}_a(h) $, we have $ |F_1 \cap F_2| < h - 1 $.  
Now, consider choosing $ h - 1 $ distinct elements $ x_1, x_2, \dots, x_{h-1} $ from $ G $. We call such a set \emph{good} if it can be extended to a member of $ \mathcal{F}_a(h) $ by adding a single element.  Otherwise, we call it a \emph{bad} set.

The equation  
$$
x_1 + x_2 + \dots + x_{h-1} + x_h = a
$$
determines a unique value for $ x_h $. However, if $ x_h \in \{x_1, x_2, \dots, x_{h-1}\} $, then this does not define a valid member of $ \mathcal{F}_a(h) $, as all elements must be distinct.  
Thus, the number of all bad $ (h-1) $-element sets is given by  
$$
\sum\limits_{x\in G} g^x_{a}(h,1).
$$
We derive that
\begin{equation}\label{Acyclicgroup3}
    f_a(h) = \frac{1}{h} \left(\binom{n}{h-1} - \sum\limits_{x\in G} g^x_{a}(h,1) \right).
\end{equation}  
Note that $g^x_{a}(h,i)$ is the number of subsets $\{x_1,\ldots,x_{h-i}\}\subseteq G$ such that $$x_1+\cdots+x_{h-i}=a-ix \text{~and~} x\in \{x_1,\ldots,x_{h-i}\}$$
for $1\leq i\leq h-1$. Equivalently, this is the number of subsets $\{x_1,\ldots,x_{h-(i+1)}\}$ satisfying
$$x_1+\cdots+x_{h-(i+1)}=a-(i+1)x \text{~and~} x\notin \{x_1,\ldots,x_{h-(i+1)}\}.$$
Notice that $f_{a - (i+1)x}(h - (i+1))$ is the number of all such subsets of size $h - (i+1)$ summing to $a - (i+1)x$,  regardless of whether they contain $x$ or not. Therefore, subtracting those that include $x$, which are counted by $g^x_{a}(h,i+1)$, yields the recurrence relation  
$$
g^x_{a}(h,i) = f_{a-(i+1)x}(h-(i+1))-g^x_{a}(h,i+1), \quad x\in G, \quad 1\leq i\leq h-2.
$$
Plugging it into \eqref{Acyclicgroup3}, we obtain the desired result.  
\qed
\end{proof}

\begin{lemma}\label{Acyclicgroup-lemma3}
Let $n\geq 4$ and $2\leq h\leq \left\lfloor\frac{n}{2}\right\rfloor+1$. Then 
\begin{align}\label{Acyclicgroup7}
    \mathop{\mathrm{max}}\limits_{a\in G}f_a(h)-\mathop{\mathrm{min}}\limits_{a\in G}f_a(h)\leq\left\{
    \begin{array}{ll}
    \frac{2^{\frac{3}{4}h}}{\prod\limits_{\substack{j=0\\j~ \text{is even} }}^{h-2}(h-j)}n^{\frac{h}{2}}, & \text{if}~ h ~\text{is even},\\[0.8cm]
    \frac{2^{\frac{3}{4}h}}{\prod\limits_{\substack{j=0\\j~ \text{is even} }}^{h-1}(h+1-j)}n^{\frac{h+1}{2}}, & \text{if}~ h~ \text{is odd}.
    \end{array}\right.
\end{align}
\end{lemma}
\begin{proof}
    We prove it by induction on $h$. Firstly, for the base case,  by Lemma \ref{Acyclicgroup-lemma2}, we have
    $$f_a(2)=\frac{1}{2}\left(\binom{n}{1}-\sum\limits_{x\in G}g_{a,x}(2,1)\right),~ f_a(3)=\frac{1}{3}\left(\binom{n}{2}-n+\sum\limits_{x\in G}g_{a,x}(3,2)\right).$$
  Recall that  $g^x_{a}(h,h-1)$ take the value 1 if  $hx=a$ and $0$ otherwise. Summing over all $x\in G$ yields
  $$0 \leq \sum\limits_{x\in G}g^x_{a}(h,h-1)\leq n,~ h\leq n-1.$$ 
  Consequently, we obtain
$$\mathop{\mathrm{max}}\limits_{a\in G}f_a(2)-\mathop{\mathrm{min}}\limits_{a\in G}f_a(2)\leq \frac{n}{2}, ~\mathop{\mathrm{max}}\limits_{a\in G}f_a(3)-\mathop{\mathrm{min}}\limits_{a\in G}f_a(3)\leq \frac{n}{3},$$
which satisfy (\ref{Acyclicgroup7}).
Now, assume that (\ref{Acyclicgroup7}) holds for all $h\leq k-1$. We proceed to prove it for $h = k$ by considering the following two cases.

\emph{Case 1:} If $k\geq 4$ is even, then 
by Lemma \ref{Acyclicgroup-lemma2}, we have
\begin{align*}\label{Acyclicgroup8}\nonumber
&\mathop{\mathrm{max}}\limits_{a\in G}f_a(k)-\mathop{\mathrm{min}}\limits_{a\in G}f_a(k)\\\nonumber
 &\leq \frac{1}{k}\left[\sum\limits_{\substack{i=2}}^{k-1}\sum_{x\in G}\left(\mathop{\mathrm{max}}\limits_{a\in G}f_{a-ix}(k-i)-\mathop{\mathrm{min}}\limits_{a\in G}f_{a-ix}(k-i)\right)+\mathop{\mathrm{max}}\limits_{a\in G}\sum\limits_{x\in G}g^x_{a}(k,k-1)\right.\\[0.2cm]\nonumber
 &\left.~~~~~~~~~~~-\mathop{\mathrm{min}}\limits_{a\in G}\sum\limits_{x\in G}g^x_{a}(k,k-1)\right] \\[0.2cm]\nonumber
&\leq \sum\limits_{\substack{i=2\\i\text{~is even}}}^{k-2}\frac{2^{\frac{3}{4}(k-i)}}{k\prod\limits_{\substack{j=0\\j~ \text{is even} }}^{k-i-2}(k-i-j)}n^{\frac{k-i}{2}+1}+\sum\limits_{\substack{i=3\\i\text{~is odd}}}^{k-1}\frac{2^{\frac{3}{4}(k-i)}}{k\prod\limits_{\substack{j=0\\j~ \text{is even} }}^{k-i-1}(k-i+1-j)}n^{\frac{k-i+3}{2}}+\frac{n}{k}\\[0.2cm]\nonumber
&=\left(\sum\limits_{\substack{i=2\\i\text{~is even}}}^{k-2}\frac{\prod\limits_{\substack{j=2\\j~ \text{is even} }}^{i-2}(k-j)}{2^{\frac{3}{4}i}n^{\frac{i-2}{2}}}
+\sum\limits_{\substack{i=3\\i\text{~is odd}}}^{k-1}\frac{\prod\limits_{\substack{j=2\\j~ \text{is even} }}^{i-3}(k-j)}{2^{\frac{3}{4}i}n^{\frac{i-3}{2}}}
+\frac{4\cdot\frac{n}{2}\cdot\prod\limits_{\substack{j=2\\j~ \text{is even} }}^{k-4}(k-j)}{2^{\frac{3}{4}k}n^{\frac{k}{2}}}\right)\frac{2^{\frac{3}{4}k}n^{\frac{k}{2}}}{\prod\limits_{\substack{j=0\\j~ \text{is even} }}^{k-2}(k-j)}\\[0.2cm]\nonumber
&\leq \left(\sum\limits_{\substack{i=2\\i\text{~is even}}}^{k-2}
\frac{(\frac{n}{2})^{\frac{i-2}{2}}}{2^{\frac{3}{4}i}n^{\frac{i-2}{2}}}
+\sum\limits_{\substack{i=3\\i\text{~is odd}}}^{k-1}
\frac{(\frac{n}{2})^{\frac{i-3}{2}}}{2^{\frac{3}{4}i}n^{\frac{i-3}{2}}}
+\frac{4\cdot(\frac{n}{2})^{\frac{k}{2}-1}}{2^{\frac{3}{4}k}n^{\frac{k}{2}}}
\right)
\frac{2^{\frac{3}{4}k}}
{\prod\limits_{\substack{j=0\\j~ \text{is even} }}^{k-2}(k-j)}n^{\frac{k}{2}}\\[0.2cm]\nonumber
&\leq \left(\sum\limits_{\substack{i=2\\i\text{~is even}}}^{k-2}\left(\frac{1}{2}\right)^{\frac{5}{4}i-1}+\sum\limits_{\substack{i=3\\i\text{~is odd}}}^{k-1}\left(\frac{1}{2}\right)^{\frac{5}{4}i-\frac{3}{2}}+\frac{1}{16}\right)\frac{2^{\frac{3}{4}k}}{\prod\limits_{\substack{j=0\\j~ \text{is even} }}^{k-2}(k-j)}n^{\frac{k}{2}}\\[0.2cm]
&\leq \left(\frac{\left(\frac{1}{2}\right)^{\frac{3}{2}}+\left(\frac{1}{2}\right)^{\frac{9}{4}}}{1-\left(\frac{1}{2}\right)^{\frac{5}{2}}}+\frac{1}{16}\right)\frac{2^{\frac{3}{4}k}}{\prod\limits_{\substack{j=0\\j~ \text{is even} }}^{k-2}(k-j)}n^{\frac{k}{2}}\\[0.2cm]
&\leq \frac{2^{\frac{3}{4}k}}{\prod\limits_{\substack{j=0\\j~ \text{is even} }}^{k-2}(k-j)}n^{\frac{k}{2}}.
\end{align*}

\emph{Case 2:} If $k\geq 5$ is odd, then by Lemma \ref{Acyclicgroup-lemma2}, we have
\begin{align*}\nonumber
&\mathop{\mathrm{max}}\limits_{a\in G}f_a(k)-\mathop{\mathrm{min}}\limits_{a\in G}f_a(k)\\\nonumber
 &\leq \frac{1}{k}\left[\sum\limits_{\substack{i=2}}^{k-1}\sum_{x\in G}\left(\mathop{\mathrm{max}}\limits_{a\in G}f_{a-ix}(k-i)-\mathop{\mathrm{min}}\limits_{a\in G}f_{a-ix}(k-i)\right)+\mathop{\mathrm{max}}\limits_{a\in G}\sum\limits_{x\in G}g^x_{a}(k,k-1)\right.\\[0.2cm]\nonumber
 &\left.~~~~~~~~~~-\mathop{\mathrm{min}}\limits_{a\in G}\sum\limits_{x\in G}g^x_{a}(k,k-1)\right] \\[0.2cm]\nonumber
&\leq \sum\limits_{\substack{i=2\\i\text{~is even}}}^{k-1}\frac{2^{\frac{3}{4}(k-i)}}{k\prod\limits_{\substack{j=0\\j~ \text{is even} }}^{k-i-1}(k-i+1-j)}n^{\frac{k-i+3}{2}}+\sum\limits_{\substack{i=3\\i\text{~is odd}}}^{k-2}\frac{2^{\frac{3}{4}(k-i)}}{k\prod\limits_{\substack{j=0\\j~ \text{is even} }}^{k-i-2}(k-i-j)}n^{\frac{k-i}{2}+1}+\frac{n}{k}\\[0.2cm]\nonumber
&=\left(\sum\limits_{\substack{i=2\\i\text{~is even}}}^{k-1}\frac{\prod\limits_{\substack{j=-1\\j~ \text{is odd} }}^{i-3}(k-j)}{k\cdot2^{\frac{3}{4}i}n^{\frac{i-2}{2}}}
+\sum\limits_{\substack{i=3\\i\text{~is odd}}}^{k-2}\frac{\prod\limits_{\substack{j=-1\\j~ \text{is odd} }}^{i-2}(k-j)}{k\cdot 2^{\frac{3}{4}i}n^{\frac{i-1}{2}}}
+\frac{4\cdot\frac{n}{2}\prod\limits_{\substack{j=-1\\j~ \text{is odd} }}^{k-4}(k-j)}{k\cdot2^{\frac{3}{4}k}n^{\frac{k+1}{2}}}\right)
\frac{2^{\frac{3}{4}k}n^{\frac{k+1}{2}}}{\prod\limits_{\substack{j=-1\\j~ \text{is odd} }}^{k-2}(k-j)}\\[0.2cm]\nonumber
&\leq\left(1+\frac{1}{k}\right)\left(\sum\limits_{\substack{i=2\\i\text{~is even}}}^{k-1}\frac{\left(\frac{n}{2}\right)^{\frac{i-2}{2}}}{2^{\frac{3}{4}i}n^{\frac{i-2}{2}}}
+\sum\limits_{\substack{i=3\\i\text{~is odd}}}^{k-2}\frac{\left(\frac{n}{2}\right)^{\frac{i-1}{2}}}{2^{\frac{3}{4}i}n^{\frac{i-1}{2}}}
+\frac{4\cdot\frac{n}{2}\cdot \left(\frac{n}{2}\right)^{\frac{k-3}{2}}}{2^{\frac{3}{4}k}n^{\frac{k+1}{2}}}\right)
\frac{2^{\frac{3}{4}k}n^{\frac{k+1}{2}}}{\prod\limits_{\substack{j=-1\\j~ \text{is odd} }}^{k-2}(k-j)}\\[0.2cm]\nonumber
&\leq \frac{6}{5}\left(\sum\limits_{\substack{i=2\\i\text{~is even}}}^{k-1}\left(\frac{1}{2}\right)^{\frac{5}{4}i-1}+\sum\limits_{\substack{i=3\\i\text{~is odd}}}^{k-2}\left(\frac{1}{2}\right)^{\frac{5}{4}i-\frac{1}{2}}+\left(\frac{1}{2}\right)^{\frac{23}{4}}\right)\frac{2^{\frac{3}{4}k}}{\prod\limits_{\substack{j=0\\j~ \text{is even} }}^{k-1}(k+1-j)}n^{\frac{k+1}{2}}\\
&\leq \frac{6}{5}\left(\frac{\left(\frac{1}{2}\right)^{\frac{3}{2}}+\left(\frac{1}{2}\right)^{\frac{13}{4}}}{1-\left(\frac{1}{2}\right)^{\frac{5}{2}}}+\left(\frac{1}{2}\right)^{\frac{23}{4}}\right)\frac{2^{\frac{3}{4}k}}{\prod\limits_{\substack{j=0\\j~ \text{is even} }}^{k-1}(k+1-j)}n^{\frac{k+1}{2}}\\
    &\leq \frac{2^{\frac{3}{4}k}}{\prod\limits_{\substack{j=0\\j~ \text{is even} }}^{k-1}(k+1-j)}n^{\frac{k+1}{2}}.
\end{align*}
That completes the proof.
\qed\end{proof}

\begin{Tproof}\textbf{~of~Theorem~\ref{Aabeliangroup-theo}.}
If $h$ is even, then by Lemma (\ref{Acyclicgroup-lemma3}), we have
\begin{equation}\label{Acyclicgroup10}
\mathop{\mathrm{max}}\limits_{a\in G}f_a\left(h\right)-\mathop{\mathrm{min}}\limits_{a\in  G}f_a\left(h\right)\leq
    \frac{2^{\frac{3}{4}h}}{\prod\limits_{\substack{j=0\\j~ \text{is even} }}^{h-2}(h-j)}n^{\frac{h}{2}}.
\end{equation}
Since $$\sum\limits_{a\in G}f_a(h)=\binom{n}{h},$$ 
it follows that $$\mathop{\mathrm{max}}\limits_{a\in  G}f_a\left(h\right)\geq \frac{1}{n}\binom{n}{h}.$$ 
Divide both sides of inequality (\ref{Acyclicgroup10}) by $\mathop{\mathrm{max}}\limits_{a\in  G}f_a\left(h\right)$. Then
\begin{align*}
\frac{\mathop{\mathrm{min}}\limits_{a\in  G}f_a\left(h\right)}{\mathop{\mathrm{max}}\limits_{a\in  G}f_a\left(h\right)}
&\geq 1-\frac{2^{\frac{3}{4}h}\cdot n^{\frac{h}{2}}}{\prod\limits_{\substack{j=0\\j~ \text{is even} }}^{h-2}(h-j)\cdot \mathop{\mathrm{max}}\limits_{a\in  G}f_a\left(h\right)}\\
&\geq 1-\frac{2^{\frac{3}{4}h}\cdot n^{\frac{h}{2}}}{\prod\limits_{\substack{j=0\\j~ \text{is even} }}^{h-2}(h-j)\cdot \frac{1}{n}\binom{n}{h}}\\
&=1-\frac{2^{\frac{3}{4}h}\cdot n^{\frac{h}{2}+1}\cdot (n-h)!\cdot\prod\limits_{\substack{j=1\\j~ \text{is odd} }}^{h-1}(h-j) }{n!}\\[0.2cm]
&\geq 1-\frac{2^{\frac{3}{4}h}\cdot n^{\frac{h}{2}+1}\cdot(n-h)!\cdot\sqrt{h!}}{n!}.
\end{align*}
Define $$X(h)=\frac{2^{\frac{3}{4}h}\cdot n^{\frac{h}{2}+1}\cdot (n-h)!\cdot\sqrt{h!}}{n!}.$$
Next, we analyze the ratio
$$\frac{X(h+1)}{X(h)}=\frac{2^{\frac{3}{4}}\sqrt{n(h+1)}}{n-h},$$
which is a monotonically increasing in $h$. 
Consequently,
$$0\leq X(h)\leq \mathop{\mathrm{max}}\left\{X(4),X\left(\left\lfloor\frac{n}{2}\right\rfloor+1\right)\right\}.$$
Without loss of generality, we assume that $n$ is even.
Then
\begin{align}\label{EDSS-equation1}\nonumber
\ln {X\left(\frac{n}{2}+1\right)}
&= \left(\frac{3n}{8}+\frac{3}{4}\right)\ln 2 + \left(\frac{n}{4}+\frac{3}{2}\right)\ln n + \ln\left(\left(\frac{n}{2}-1\right)!\right)\\
&+\frac{1}{2}\ln\left(\left(\frac{n}{2}+1\right)!\right) - \ln(n!)
\end{align}
Using Stirling's formula, we have
$$\ln\left(\left(\frac{n}{2}-1\right)!\right)
= \frac{n}{2}\ln n - \frac{n}{2}\ln 2 - \frac{n}{2}+ o\left(n\right),
$$
$$
\frac{1}{2}\ln\left(\left(\frac{n}{2}+1\right)!\right) 
= \frac{n}{4}\ln n - \frac{n}{4}\ln 2 - \frac{n}{4}+ o\left(n\right),
$$
and 
$$
\ln(n!) = n\ln n - n + o\left(n\right). 
$$
Substituting these into (\ref{EDSS-equation1}) yields
$$\ln {X\left(\frac{n}{2}+1\right)}=\left(\frac{1}{4} - \frac{3}{8}\ln 2\right)n+o(n).$$
Since $\frac{1}{4} - \frac{3}{8}\ln 2\approx-0.00993<0$,
we have
$$\lim\limits_{n\rightarrow \infty}X\left(\frac{n}{2}+1\right)=\lim\limits_{n\rightarrow \infty}\exp\left(\ln {X\left(\frac{n}{2}+1\right)}\right)=0.$$
The same conclusion holds when $n$ is odd. Moreover,
$$\lim\limits_{n\rightarrow \infty}X(4)=\lim\limits_{n\rightarrow \infty}\frac{2^3\cdot n^3\cdot \sqrt{4!}}{n(n-1)(n-2)(n-3)}=0.$$
Thus, for all $4\leq h\leq \left\lfloor\frac{n}{2}\right\rfloor+1$, we have
$$\lim\limits_{n\rightarrow \infty}X(h)=0.$$ 
That is, 
$$\lim\limits_{n\rightarrow \infty}\frac{\mathop{\mathrm{min}}\limits_{a\in  G}f_a\left(h\right)}{\mathop{\mathrm{max}}\limits_{a\in  G}f_a\left(h\right)}\geq  1,~ 4\leq h\leq \left\lfloor\frac{n}{2}\right\rfloor+1.$$
Clearly, the fraction is at most $1$, which completes the proof.

If $h$ is odd, the proof follows analogously to the even case, and thus the details are omitted.
\qed\end{Tproof}

\section*{Acknowledgements}
The author greatly appreciate Gyula O.H. Katona for proposing the question and for his insightful discussions. The author is also grateful to P\'{e}ter P\'{a}l Pach for his helpful suggestions and to the anonymous referees for their valuable comments. The work is supported by China Scholarship Council (No.202406290013), the National Natural Science Foundation of China (No.12271439, No.12371358) and the Guangdong Basic and Applied Basic Research Foundation (No.2023A1515010986).

\end{document}